\documentclass[12pt]{amsart}

\usepackage{graphicx}
\usepackage{amsthm}
\usepackage{amsmath, amssymb}
\swapnumbers

\theoremstyle{plain}
\newtheorem{Thm}{Theorem}

\newtheorem{Coro}[Thm]{Corollary}
\newtheorem{Lem}[Thm]{Lemma}

\theoremstyle{definition}
\newtheorem{Def}[Thm]{Definition}

\begin{document}

\title{Flipping and stabilizing Heegaard splittings}

\author{Jesse Johnson}
\address{\hskip-\parindent
        Department of Mathematics \\
        Yale University \\
        PO Box 208283 \\
        New Haven, CT 06520 \\
        USA}
\email{jessee.johnson@yale.edu}


\thanks{Research supported by NSF MSPRF grant 0602368}

\begin{abstract}
We show that the number of stabilizations needed to interchange the handlebodies of a Heegaard splitting of a closed 3-manifold by an isotopy is bounded below by the smaller of twice its genus or half its Hempel distance.  This is a combinatorial version of a proof by Hass, Thompson and Thurston of a similar theorem, but with an explicit bound in terms of distance.  We also show that in a 3-manifold with boundary, the stable genus of a Heegaard splitting and a boundary stabilization of itself is bounded below by the same value.
\end{abstract}

\maketitle

\section{Introduction}

A \textit{Heegaard splitting} for a compact, connected, closed, orientable 3-manifold $M$ is a triple $(\Sigma, H^-, H^+)$ where $\Sigma$ is a compact, separating surface in $M$ and $H^-$, $H^+$ are handlebodies in $M$ such that $M = H^- \cup H^+$ and $\partial H^- = \Sigma = H^- \cap H^+ = \partial H^+$.  A \textit{stabilization} of $(\Sigma, H^-, H^+)$ is a new Heegaard splitting constructed by taking a connect sum of $(\Sigma, H^-, H^+)$ with a Heegaard splitting for $S^3$.  (This will be described more carefully later.)

A Heegaard splitting is \textit{flippable} if there is an isotopy of $M$ that takes $\Sigma$ to itself but interchanges the two handlebodies or, equivalently, if there is an isotopy taking the oriented surface $\Sigma$ to itself with the opposite orientation.  Whether or not a Heegaard splitting is flippable, it will always have a stabilization that is flippable.  The \textit{flip genus} of a Heegaard splitting is the genus of the smallest stabilization that is flippable.  

\begin{Thm}
\label{mainthm1}
Given a genus $k \geq 2$ Heegaard splitting $(\Sigma, H^-, H^+)$ for a closed 3-manifold $M$, the flip genus of $\Sigma$ is greater than or equal to $\min\{2k, \frac{1}{2} d(\Sigma)\}$.
\end{Thm}

Here, $d(\Sigma)$ is defined as follows: The \textit{curve complex} $C(\Sigma)$ of a compact surface $\Sigma$ is a simplicial complex whose vertices are isotopy classes of essential simple closed curves in $\Sigma$ and whose simplices are pairwise disjoint sets of loops.  The distance $d(\ell^-,\ell^+)$ between simple closed curves $\ell^-,\ell^+$ in $\Sigma$ is defined as the length of the shortest edge path in $C(\Sigma)$ between the vertices that represent them.  The \textit{(Hempel) distance} $d(\Sigma)$ of a Heegaard splitting $(\Sigma, H^-, H^+)$ is the minimum of $d(\ell^-, \ell^+)$ over all pairs such that $\ell^-$ bounds a disk in $H^-$ and $\ell^+$ bounds a disk in $H^+$.

Hass, Thompson and Thurston~\cite{htt:stabs} recently proved that there exist Heegaard splittings with flip genus equal to $2k$.  (A straightforward construction shows that the flip genus is never more than $2k$.)  They construct examples by gluing the handlebodies together by a high power of a pseudo-Anosov map.  Theorem~\ref{mainthm1} implies their result because, as proved by Hempel~\cite{Hempel:complex}, such Heegaard splittings will have distance greater than $4k$.

The \textit{stable genus} of two Heegaard splittings $(\Sigma, H^-, H^+)$ and $(\Sigma', H'^-, H'^+)$ is the genus of the smallest stabilization of $\Sigma$ that is isotopic to a stabilization of $\Sigma'$.  For this definition, we will not pay attention to the names of the two handlebodies or the orientations of the surfaces.  We just want to calculate when the stabilizations will be isotopic as unoriented surfaces.  In this case, the methods used to prove Theorem~\ref{mainthm1} cannot be used to bound stable genus in closed manifolds.  However, they can be used for 3-manifolds with boundary.   (Heegaard splittings for manifolds with non-empty boundary will be defined in a later section.)

\begin{Thm}
\label{mainthm2}
Let $\Sigma$ be a genus $k \geq 2$ Heegaard splitting of a 3-manifold with a single boundary component and let $\Sigma'$ be the result of boundary stabilizing $\Sigma$.   Then the stable genus of $\Sigma$ and $\Sigma'$ is greater than or equal to $\min \{ 2k, \frac{1}{2} d(\Sigma)\}$.  
\end{Thm}

Moriah and Sedgwick~\cite{morsedg} have asked whether there is either a closed 3-manifold or a 3-manifold with a single boundary component that has a weakly reducible Heegaard splitting that is not minimal genus.  Examples are known for more boundary components.  They have suggested boundary stabilization (which always produces weakly reducible Heegaard splittings) as a possible way to construct examples with one boundary component.  Theorem~\ref{mainthm2} implies the following:

\begin{Coro}
\label{maincoro}
If a 3-manifold $M$ with a single genus $b$ boundary component has a genus $k \geq 2$ Heegaard surface $\Sigma$ such that $d(\Sigma) > 2(k + b)$ then a boundary stabilization of $\Sigma$ is an irreducible, weakly reducible Heegaard splitting of non-minimal genus.
\end{Coro}

A future paper will deal with the problem of bounding from below the stable genus of Heegaard splittings in a closed 3-manifold by generalizing the methods discussed here.

I would like to thank Joel Hass and Abby Thompson for discussing their proof with me during the AIM workshop on Heegaard splittings, triangulations and hyperbolic geometry, held in December 2007, and Andrew Casson for helping me to work out the details of the proof below.

\section{Sweep-outs and graphics}
\label{sweepsect}

A \textit{handlebody} is a connected 3-manifold that is homeomorphic to a regular neighborhood of a graph embedded in $S^3$.  Given a properly embedded graph $K$ in a 3-manifold $M$ with boundary (i.e. one or more of the vertices may be in the boundary of $M$, but the interiors of the edges are in the interior of $M$), a \textit{compression body} is a connected 3-manifold homeomorphic to a regular neighborhood $H$ of the union of $K$ and every boundary component that contains a vertex of $K$.  The union of $K$ and the boundary components is called a \textit{spine} for $H$.  Note that a handlebody is also a compression body.

The subset $\partial H \cap \partial M$ of $\partial H$ is called the \textit{negative boundary}, written $\partial_- H$, and the remaining component is the \textit{positive boundary}, $\partial_+ H$.  For a 3-manifold $M$ with boundary, a Heegaard splitting for $M$ is a triple $(\Sigma, H^-, H^+)$ where $\Sigma \subset M$ is a closed, embedded surface, $H^-, H^+ \subset M$ are compression bodies, $H^- \cup H^+ = M$, and $\partial_+ H^- = \Sigma = \partial_+ H^+ = H^- \cap H^+$.  It follows that $\partial M = \partial_- H^- \cup \partial_- H^+$.

A \textit{sweep-out} is a smooth function $f : M \rightarrow [-1,1]$ such that for each $x \in (-1,1)$, the level set $f^{-1}(x)$ is a closed surface.  Moreover, $f^{-1}(-1)$ must be the union of a graph and some number of boundary components while $f^{-1}(1)$ is the union of a second graph and the remaining boundary components.  Each of $f^{-1}(-1)$ and $f^{-1}(1)$ is called a \textit{spine} of the sweep-out.  Each level surface of $f$ is a Heegaard surface for $M$.  The spines of the sweep-outs are spines of the two compression bodies in the Heegaard splitting.  

Conversely, given a Heegaard splitting $(\Sigma, H^-, H^+)$ for $M$, there is a sweep-out for $M$ such that each level surface is isotopic to $\Sigma$.  We will say that a sweep-out \textit{represents} $(\Sigma, H^-, H^+)$ if $f^{-1}(-1)$ is isotopic to a spine for $H^-$ and $f^{-1}(1)$ is isotopic to a spine for $H^+$.  The level surfaces of such a sweep-out will be isotopic to $\Sigma$.  A simple construction (which will be left to the reader) implies the following:

\begin{Lem}
Every Heegaard splitting of a compact, connected, orientable, smooth 3-manifold is represented by a sweep-out.
\end{Lem}

By definition, if two Heegaard splittings are represented by the same sweep-out then they are isotopic.  If two sweep-outs represent the same Heegaard splitting then after a sequence of edge slides on the spines of the sweep-outs, the sweep-outs will be isotopic.

A \textit{stable function} between smooth manifolds $M$ and $N$ is a smooth function $\phi : M \rightarrow N$ such that in the space $C^\infty(M,N)$ of smooth functions from $M$ to $N$, there is a neighborhood $N$ around $\phi$ such that each function in $N$ is isotopic to $\phi$.  A Morse function is a smooth function from a smooth manifold to $\mathbf{R}$ and one can think of stable functions as a generalization of Morse theory to functions whose ranges have dimension greater than one.

Given two sweep-outs, $f$ and $g$, their product is a smooth function $f \times g : M \rightarrow [-1,1] \times [-1,1]$.  (That is, we define $(f \times g)(x) = (f(x),g(x))$.)  Kobayashi~\cite{Kob:disc} has shown that after an isotopy of $f$ and $g$, we can assume that $f \times g$ is a stable function on the complement of the four spines.  The local behavior of stable functions between dimensions two and three has been classified~\cite{mather} and coincides with the classification by Cerf~\cite{cerf:strat} that was used by Rubinstein and Scharlemann~\cite{rub:compar}, who first used pairs of sweep-outs to compare Heegaard splittings.

At each point in the complement of the spines, the differential of the map $f \times g$ is a linear map from $\mathbf{R}^3$ to $\mathbf{R}^2$.  This map will have a one dimensional kernel for a generic point in $M$.  The \textit{discriminant set} for $f \times g$ is the set of points where the derivative has a higher dimensional kernel.  (In these dimensions, all the critical points in a stable function have two dimensional kernels.)  Mather's classification of stable functions~\cite{mather} implies that the discriminant set in this case will be a one dimensional smooth submanifold in the complement in $M$ of the spines.  It consists of all the points where a level surface of $f$ is tangent to a level surface of $g$.  Some examples are shown in Figure~\ref{grfig}.  (For a more detailed description see~\cite{Kob:disc} or~\cite{rub:compar}.)  
\begin{figure}[htb]
  \begin{center}
  \includegraphics[width=4in]{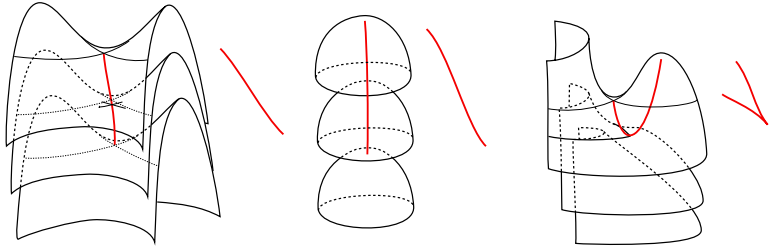}
  \caption{Edges of the graphic are formed by points where the level surface of $f$ are tangent to level surfaces of $g$.  Here, $f$ is the height function and its level surfaces are horizontal planes}
  \label{grfig}
  \end{center}
\end{figure}

The function $f \times g$ sends the discriminant to a graph in $[-1,1] \times [-1,1]$ called the \textit{Rubinstein-Scharlemann graphic} (or just the \textit{graphic} for short).  The parts of the graphic corresponding to the tangencies in Figure~\ref{grfig} are shown next to the surfaces.  The vertices in the interior of the graphic are valence four (crossings) or valence two (cusps).  The vertices in the boundary are valence one or two.

The pre-image in $f \times g$ of an arc $[-1,1] \times \{s\}$ is the level set $g^{-1}(s)$ and the restriction of $f$ to this level surface is a function $\phi_s$ with critical points in the levels where the arc $[-1,1] \times \{s\}$ intersects the graphic as well as possibly at the levels $-1$ and/or $1$.  The same is true if we switch $f$ and $g$.  

\begin{Def}
The function $f \times g$ is \textit{generic} if $f \times g$ is stable and each arc $\{t\} \times [-1,1]$ or $[-1,1] \times \{s\}$ contains at most one vertex of the graphic.
\end{Def}

If the arc does not intersect any vertices then every critical point of $\phi_s$ will be non-degenerate and away from $-1$ and $1$ no two critical points will be in the same level.  In other words, $\phi_s$ will be \textit{Morse away from} $-1$ and $1$.  If the arc passes through a vertex then in the levels other than $-1$ and $1$, $\phi_s$ will either have a degenerate critical point or two non-degenerate critical points at the same level.  We will say that such a $\phi_s$ is \textit{near-Morse away from} $-1$ and $1$.

\section{Spanning Heegaard surfaces}
\label{facingsect}

Let $f$ and $g$ be sweep-outs representing Heegaard splittings $(\Sigma, H^-, H^+)$ and $(\Sigma',H'^-,H'^+)$, respectively.  For each $s \in (-1,1)$, define $\Sigma'_s = g^{-1}(s)$, $H'^-_s = g^{-1}([-1,s])$ and $H'^+_s = g^{-1}([s,1])$.  Similarly, for $t \in (-1,1)$, define $\Sigma_t = f^{-1}(t)$.  We will say that $\Sigma_t$ is \textit{mostly above} $\Sigma'_s$ if each component of $\Sigma_t \cap H'^-_s$ is contained in a disk subset of $\Sigma_t$.  Similarly, $\Sigma_t$ is \textit{mostly below} $g_t$ if each component of $\Sigma_t \cap H'^+_s$ is contained in a disk in $\Sigma_t$.  

\begin{Def}
\label{ffdef}
We will say $g$ \textit{spans} $f$ if there are values $s, t_-, t_+ \in [-1,1]$ such that $\Sigma_{t_-}$ is mostly below $\Sigma'_s$ while $\Sigma_{t_+}$ is mostly above $\Sigma'_s$.  We will say that $g$ spans $f$ \textit{positively} if $t_- < t_+$ or \textit{negatively} if $t_- > t_+$.
\end{Def}

We can understand spanning in terms of the graphic as follows:  Let $R_a \subset (-1,1) \times (-1,1)$ be the set of all values $(t,s)$ such that $\Sigma_t$ is mostly above $\Sigma'_s$.  Let $R_b \subset (-1,1) \times (-1,1)$ be the set of all values $(t,s)$ such that $\Sigma_t$ is mostly below $\Sigma'_s$.  For a fixed $t$, there will be values $a,b$ such that $\Sigma_t$ will be mostly above $\Sigma'_s$ if and only if $s \in [-1,a)$ and mostly above $\Sigma'_s$ if and only if $s \in [b,1]$, as shown in Figure~\ref{mostlyfig}.
\begin{figure}[htb]
  \begin{center}
  \includegraphics[width=3.5in]{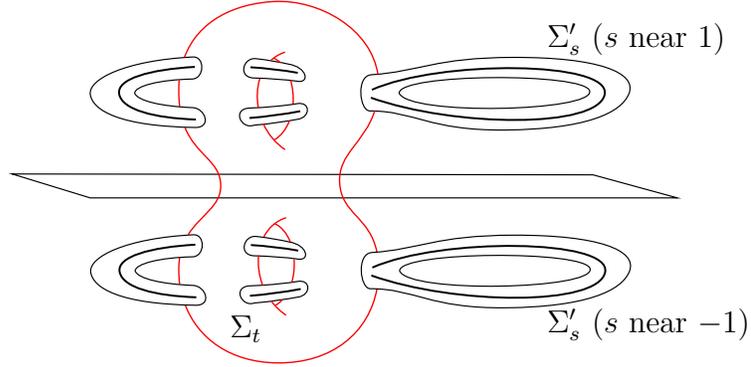}
  \put(-170,10){$\Sigma_t$}
  \put(-50,12){$\Sigma'_s$ ($s$ near $-1$)}
  \put(-50,120){$\Sigma'_s$ ($s$ near 1)}
  \caption{The vertical surface represents a level set $\Sigma_t$ of $f$, while the horizontal surfaces represent level surfaces $\Sigma'_s$ of $g$.  Given a fixed $t$ $\Sigma_t$ will be mostly above $\Sigma'_s$ for small $s$ and mostly above $\Sigma'_s$ for large $s$.}
  \label{mostlyfig}
  \end{center}
\end{figure}

Thus the regions $R_a$ and $R_b$ will be vertically convex, as in Figure~\ref{spanningfig}.  The sweep-out $g$ will span $f$ if in the graphic for $f \times g$, there is a horizontal arc that intersects both $R_a$ and $R_b$.  The figure also shows examples of pairs of sweep-outs that don't span, or that span with both signs.
\begin{figure}[htb]
  \begin{center}
  \includegraphics[width=2.5in]{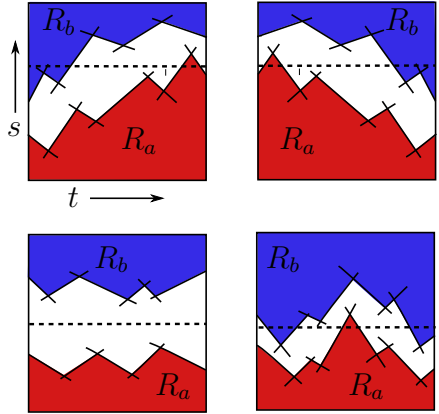}
  \put(-150,145){$R_b$}
  \put(-20,145){$R_b$}
  \put(-130,55){$R_b$}
  \put(-65,55){$R_b$}
  \put(-120,100){$R_a$}
  \put(-60,100){$R_a$}
  \put(-107,7){$R_a$}
  \put(-34,7){$R_a$}
  \put(-163,105){$s$}
  \put(-140,78){$t$}
  \caption{Clockwise from the top left, the graphics correspond to pairs of sweep-outs $f$, $g$ such that (1) $g$ spans $f$ positively, (2) $g$ spans $f$ negatively, (3) $g$ spans $f$ with both signs and (4) $g$ does not span $f$.  The dotted line represents the arc $[-1,1] \times \{s\}$.}
  \label{spanningfig}
  \end{center}
\end{figure}

Note that if $g$ spans $f$ positively then $-g$ will span $f$ negatively and $g$ will span $-f$ negatively.  

\begin{Lem}
\label{regionsboundedlem}
The closure of $R_a$ in $(-1,1) \times (-1,1)$ is bounded by arcs of the Rubinstein-Scharlemann graphic, as is the closure of $R_b$.
\end{Lem}

\begin{proof}
To see this, note that for fixed $t$, the restriction of $g$ to $\Sigma_t$ has singular points at precisely the levels where the arc $\{t\} \times [-1,1]$ intersect the graphic.  The intersection of $\{t\} \times [-1,1]$ with $R_a$ is an arc of the form $\{t\} \times [-1,s_a)$ where $s_a$ is the smallest critical point such that $g^{-1}([-1,s_a]) \cap \Sigma_t$ is not contained in a disk in $\Sigma_t$.  Thus $R_a$ is the region of $(-1,1) \times (-1,1)$ bounded above by some collection of arcs in the graphic.  The same argument can be applied to $R_b$.
\end{proof}

\begin{Def}
We will say that $(\Sigma', H'^-, H'^+)$ \textit{spans} $(\Sigma, H^-, H^+)$ \textit{positively} or \textit{negatively} if there are sweep-outs $f$ and $g$ representing $(\Sigma, H^-, H^+)$ and $(\Sigma', H'^-, H'^+)$, respectively, such that $g$ spans $f$ positively or negatively, respectively.
\end{Def}

Here, the convention that $f^{-1}(-1)$ is a spine of $H^-$ and $f^{-1}(1)$ is a spine for $H^+$ is important.  If $f$ represents $(\Sigma, H^-, H^+)$ then the sweep-out $-f$ will represent $(\Sigma, H^+, H^-)$.  Thus if $(\Sigma', H'^-, H'^+)$ spans $(\Sigma, H^-, H^+)$ positively then $(\Sigma', H'^+, H'^-)$ will span $(\Sigma, H^-, H^+)$ negatively and $(\Sigma', H'^-, H'^+)$ will span $(\Sigma, H^+, H^-)$ negatively.  In particular, if $(\Sigma', H'^-, H'^+)$ spans $(\Sigma, H^-, H^+)$ and $\Sigma'$ is flippable then $(\Sigma', H'^-, H'^+)$ will span $(\Sigma, H^-, H^+)$ with both signs.

As an example, note the following Lemma:

\begin{Lem}
\label{withitselflem}
Every Heegaard splitting spans itself positively.  The Heegaard splitting $(\Sigma, H^+, H^-)$ spans $(\Sigma, H^-, H^+)$ negatively.
\end{Lem}

\begin{proof}
Let $(\Sigma, H^-, H^+)$ be a Heegaard splitting represented by a sweep-out $f$ such that $f^{-1}(0)$ is the surface $\Sigma$.  Let $\phi$ be any Morse function on $\Sigma$ whose image is contained in $[-\frac{1}{2},\frac{1}{2}]$.  Identify $f^{-1}([-\frac{1}{2},\frac{1}{2}])$ with $\Sigma \times [-\frac{1}{2},\frac{1}{2}]$ such that $f^{-1}(t) = \Sigma \times \{t\}$ for each $t \in [-\frac{1}{2},\frac{1}{2}]$.

Let $\Sigma'$ be the graph of $\phi$ in $f^{-1}([-\frac{1}{2},\frac{1}{2}]) = \Sigma \times [-\frac{1}{2},\frac{1}{2}]$, i.e. the set of points $\{(x,\phi(x)) | x \in \Sigma\}$.  The surface $\Sigma'$ is isotopic to the level surface $\Sigma \times 0 = \Sigma$ so $\Sigma'$ determines a Heegaard splitting $(\Sigma', H'^-, H'^+)$ such that $f^{-1}(-1)$ is contained in $H'^-$ and $f^{-1}(1)$ is contained in $H'^+$.

There is a sweep-out $g$ representing $(\Sigma', H'^-, H'^+)$ such that $g^{-1}(0) = \Sigma'$.  Then $H'^-_0 = H'^-$ and $H'^+_0 = H'^+$.  For $t < -\frac{1}{2}$, $\Sigma_t$ is contained in $H'^-_0$, so $\Sigma_t$ is mostly below $\Sigma'_{\frac{1}{2}}$.  For $t > 0$, $\Sigma_t$ is contained in $H'^+_0$ so $\Sigma_t$ is mostly above $\Sigma'_0$.  Thus $g$ spans $f$ positively.  Since $(\Sigma', H'^-, H'^+)$ is isotopic to $(\Sigma, H^-, H^+)$, the sweep-out $g$ also represents $(\Sigma, H^-, H^+)$, so this Heegaard splitting spans itself positively.  Combining this example with the fact that switching the order of the handlebodies reverses the direction of the spanning implies the second half of the Lemma.
\end{proof}

\section{Stabilization}
\label{stablsect}

Let $(\Sigma, H^-, H^+)$ be a Heegaard splitting of a 3-manifold $M$ and $(\Xi, G^-, G^+)$ a Heegaard splitting of the 3-sphere, $S^3$.  Let $B \subset M$ be an open ball such that $B \cap \Sigma$ is a single open disk.  The intersection of $\Sigma$ with the boundary of the closure $\bar B$ is a simple closed curve.  Similarly, let $B' \subset S^3$ be an open ball with $B' \cap \Xi$ an open disk.  

The complement in $S^3$ of $B'$ is a closed ball and we would like to identify this ball with the closure in $M$ of $B$.  Let $\phi : S^3 \setminus B' \rightarrow \bar B$ be a homeomorphism that sends the loop $\Xi \cap \partial \bar B'$ onto the loop $\Sigma \cap \partial \bar B$, sends $G^- \cap \partial \bar B'$ onto $H^- \cap \partial \bar B$ and sends $G^+ \cap \partial \bar B'$ onto $H^+ \cap \partial \bar B$.  

Let $\Sigma^*$ be the union of $\Sigma \setminus B$ and the image $\phi(\Xi \setminus B')$.  Similarly, define $H^{*-} = (H^- \setminus B) \cup \phi(G^- \setminus B')$ and $H^{*+} = (H^+ \setminus B) \cup \phi(G^+ \setminus B')$.  The set $H^{*-}$ is a union of two handlebodies that intersect a disk, so $H^{*-}$ is itself a handlebody.  The same reasoning implies $H^{*+}$ is also a handlebody, so $(\Sigma^*, H^{*-}, H^{*+})$ is a Heegaard splitting of $M$.  Note that the surface $\Sigma^*$ has genus equal to the genus of $\Sigma$ plus the genus of $\Xi$.

\begin{Def}
The Heegaard splitting $(\Sigma^*, H^{*-}, H^{*+})$ constructed above is called a \textit{stabilization} of $(\Sigma, H^-, H^+)$.  
\end{Def}

Note that if we take $(\Xi, G^-, G^+)$ to be the genus zero Heegaard splitting, then $(\Sigma^*)$ is isotopic to $\Sigma$.  Thus every Heegaard splitting is a (trivial) stabilization of itself.

By Waldhausen's Theorem~\cite{wald:sphere}, two Heegaard splittings of $S^3$ are isotopic if and only if they have the same genus.  The balls $B$ and $B'$ are unique up to isotopies of $M$ and $M'$, respectively, preserving $\Sigma$ and $\Xi$, respectively.  Thus two stabilizations of the same Heegaard splitting are isotopic if and only if they have the same genus.

\begin{Lem}
\label{stabfacinglem}
If a Heegaard splitting $(\Sigma', H'^-, H'^+)$ spans a second Heegaard splitting $(\Sigma, H^-, H^+)$ positively then every stabilization of $(\Sigma', H'^-, H'^+)$ spans $(\Sigma, H^-, H^+)$ positively.  If $(\Sigma', H'^-, H'^+)$ spans $(\Sigma, H^-, H^+)$ negatively then every stabilization of $(\Sigma', H'^-, H'^+)$ spans $(\Sigma, H^-, H^+)$ negatively.
\end{Lem}

\begin{proof}
Since $(\Sigma', H'^-, H'^+)$ spans $(\Sigma, H^-, H^+)$ positively, let $f$ and $g$ be sweep-outs for $(\Sigma, H^-, H^+)$ and $(\Sigma', H'^-, H'^+)$, respectively, such that $g$ spans $f$ positively.  Let $s, t_-, t_+ \in [-1,1]$ be values such that, $\Sigma_{t_-}$ is mostly below $\Sigma'_s$ and $\Sigma_{t_+}$ is mostly above $\Sigma'_s$.

Replace $(\Sigma', H'^-, H'^+)$ by the isotopic Heegaard splitting whose Heegaard surface is $\Sigma'_s = g^{-1}(s)$.  Let $B \subset M$ be an open ball as above whose closure is contained in $f^{-1}(t_-,t_+)$.  Let $B' \subset S^3$ be an open ball that intersects a Heegaard splitting for $S^3$ of the appropriate genus in an open ball.  Let $(\Sigma^*, H^{*-}, H^{*+})$ be the stabilization of $(\Sigma', H'^-, H'^+)$ constructed by identifying $S^3 \setminus B'$ with $\bar B$.  Let $g^*$ be a sweep-out such that $(g^*)^{-1}(0) = \Sigma^*$.

The complement in $B$ of $\Sigma'$ is equal (as a set) to the complement of $\Sigma^*$, and the same is true for the corresponding compression bodies in the Heegaard splittings.  Since $\Sigma_{t_-}$ is mostly below $\Sigma'_s = \Sigma'$ and disjoint from $B$, $\Sigma_{t_-}$ is also mostly below $\Sigma^*_0 = (g^*)^{-1}(0) = \Sigma^*$.  Similarly, $\Sigma_{t_+}$ is mostly above $\Sigma^*_0$, so $(\Sigma^*, H^{*-}, H^{*+})$ spans $(\Sigma, H^-,H^+)$ with the same sign as $(\Sigma', H'^-, H'^+)$.
\end{proof}

Consider a 3-manifold $M$ with a single boundary component.  In this situation, we will adopt the convention that if $(\Sigma, H^-, H^+)$ is a Heegaard splitting for $M$ then $H^-$ is a handlebody and $H^+$ is a compression body.  We can decompose $H^+$ into $\partial M \times [-1,1]$ and a collection of 1-handles.  Let $\alpha$ be a vertical arc in $\partial M \times [-1,1]$ disjoint from the 1-handles.

A regular neighborhood $N$ in $H^+$ of $\alpha \cup \partial M$ is homeomorphic to $\partial M \times [0,1]$ and intersects $\Sigma$ in a single disk.  Thus $H^- \cup N$ is a compression body.  The complement in $\partial M \times [0,1]$ of $N$ is homeomorphic to the complement in $\partial M \times [0,1]$ of a regular neighborhood of $\alpha$.  This is a punctured surface cross an interval, which is homeomorphic to a handlebody.  Thus $H^+ \setminus N$ is homeomorphic to the union of a handlebody and a collection of 1-handles.  This union is a handlebody so $H^- \cup N$ and $H^+ \setminus N$ determine a Heegaard splitting.

\begin{Def}
A \textit{boundary stabilization} of $(\Sigma, H^-, H^+)$ is the Heegaard splitting $(\Sigma^*, H^{*-}, H^{*+})$ where $H^{*-} = H^+ \setminus N$, $H^{*+} = H^- \cup \bar N$ and $\Sigma^*$ is their common boundary.
\end{Def}

Note that we have labeled $H^{*-}$ and $H^{*+}$ so as to keep the convention that $H^{*-}$ is a handlebody and $H^{*+}$ is a compression body.  The genus of a boundary stabilization is equal to the genus of the boundary plus the genus of the original Heegaard splitting.  The isotopy class of the boundary stabilization is determined by the vertical arc $\alpha$.  Such an arc is unique up to isotopy so any two boundary stabilizations of the same Heegaard splitting are isotopic.

\begin{Lem}
\label{bstabfacinglem}
If $(\Sigma', H'^-, H'^+)$ spans $(\Sigma, H^-, H^+)$ positively then a boundary stabilization of $(\Sigma', H'^-, H'^+)$ spans $(\Sigma, H^-, H^+)$ negatively.
\end{Lem}

\begin{proof}
Let $f$ and $g$ be sweep-outs representing $(\Sigma, H^-, H^+)$ and $(\Sigma', H'^-, H'^+)$, respectively.  Let $s,t_-,t_+ \in [-1,1]$ be as in Definition~\ref{ffdef}.  Replace $(\Sigma', H'^-, H'^+)$ with the isotopic Heegaard splitting whose Heegaard surface is $g^{-1}(s)$.  Thus $H'^-_s = H'^-$ and $H'^+_s = H'^+$.  

Let $\alpha \subset H'^+$ be an arc defining a boundary stabilization of $(\Sigma', H'^-, H'^+)$, and $N$ a regular neighborhood in $H^+$ of $\alpha \cup \partial M$.  We can choose this regular neighborhood such that $N \cap \Sigma_{t_-}$ and $N \cap \Sigma_{t_+}$ is each a collection of disks.

Let $(\Sigma^*, H^{*-}, H^{*+})$ be the boundary stabilization determined by $\alpha$ and $N$.  Let $g^*$ be a sweep-out representing $(\Sigma^*, H^{*-}, H^{*+})$ such that $ \Sigma^* = \Sigma^*_0 = (g^*)^{-1}(0)$.  The intersection of $\Sigma_{t_-}$ with $H^{*-}_0$ is the union of $\Sigma_{t_-} \cap H^+_0$ and a collection of disjoint disks.  Since $\Sigma_{t_-}$ is mostly below $\Sigma'_0$, it is mostly above $\Sigma^*_0$.  Similarly, $\Sigma_{t_+}$ is mostly below $\Sigma^*_0$, so $g^*$ spans $f$ negatively.
\end{proof}

\section{Compressing Heegaard surfaces}
\label{facingbothsect}

We have seen that every stabilization or boundary stabilization of a Heegaard splitting spans the original.  In this section we will prove the converse.  Let $f$ and $g$ be sweep-outs representing Heegaard splittings $(\Sigma, H^-, H^+)$ and $(\Sigma', H'^-, H'^+)$, respectively, with genera $k$ and $k'$, respectively.

\begin{Lem}
\label{bothfacinglem}
Assume $M$ is irreducible.  If $g$ spans $f$ then $\Sigma'$ is an amalgamation along $\Sigma$.  If $g$ spans $f$ both positively and negatively then $\Sigma'$ is an amalgamation along a union of two copies of $\Sigma$ such that $k' \geq 2k$.
\end{Lem}

By amalgamation, we mean the following:  Let $F \subset M$ be a separating surface and let $(\Sigma^\#, H^{\#-}, H^{\#+})$ and $(\Sigma^*, H^{*-}, H^{*+})$ be Heegaard splittings for the closures of the components of $M \setminus F$.  These determine a handle decomposition for $M$ in which some of the 2-handles are added before some of the 1-handles.  If we rearrange the order of the handles, we can produce a Heegaard splitting $(\Sigma', H'^-, H'^+)$ for all of $M$, as in Figure~\ref{amalgfig}.  We will say that the resulting Heegaard surface $\Sigma'$ is an \textit{amalgamation along $F$ of $\Sigma^\#$ and $\Sigma^*$}.  See~\cite{sch:cross} for a more detailed description of this construction.
\begin{figure}[htb]
  \begin{center}
  \includegraphics[width=2.5in]{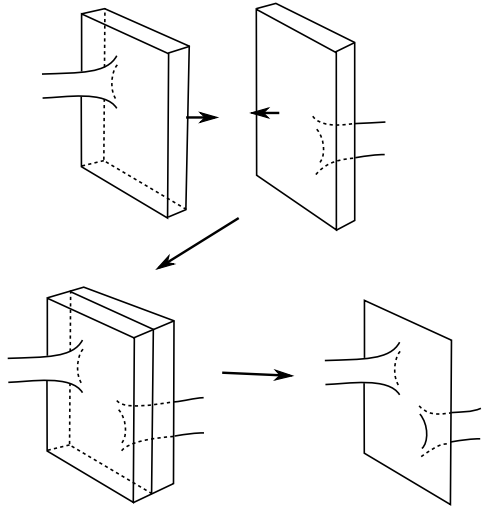}
  \caption{Amalgamating two Heegaard splittings along a separating surface.}
  \label{amalgfig}
  \end{center}
\end{figure}

The classification of Heegaard splittings of handlebodies and compression bodies~\cite[Corollary 2.12]{st:crossi} implies that if $\Sigma'$ is an amalgamation along $\Sigma$ then $\Sigma'$ is a stabilization of either $\Sigma$ or a boundary stabilization of $\Sigma$.  We will prove Lemma~\ref{bothfacinglem} as a corollary of Lemmas~\ref{compresstoamalglem},~\ref{crossisepsurfacelem} and~\ref{onesthreetslem}, the first of which gives a sufficient condition for determining when $\Sigma'$ is an amalgamation along a surface $F$.

\begin{Lem}
\label{compresstoamalglem}
Let $(\Sigma', H'^-, H'^+)$ be a Heegaard splitting for an irreducible 3-manifold $M$ and let $\Sigma'=S_0, S_1, \dots, S_n = F$ be a sequence of surfaces such that up to isotopy, each $S_{i+1}$ is the result of compressing $S_i$ along a disk $D_i$ properly embedded in the complement of $S_i$.  Then $\Sigma'$ is an amalgamation along $F$.
\end{Lem}

\begin{proof}
Let $D_0$ be the compressing disk for $S_0$ such that compressing $S_0$ across $D_0$ produces $S_1$.  Without loss of generality, assume $D_0$ is contained in $H'^-$.  Let $N_0$ be the union of a regular neighborhood in $H'^-$ of $D_0$ and a regular neighborhood in $H'^+$ of $\Sigma$.  This set is a compression body with $\partial_- N_0 = S_1$ and $\partial_+ N_0$ a surface parallel to $\Sigma$.  

The complement in $H'^+$ of $N_0$ is a handlebody, so $N_0$ determines a Heegaard splitting for one component of the complement of $S_1$.  The other component of $M \setminus S_1$ is the handlebody $H^- \setminus N_0$.  This component has a Heegaard splitting consisting of a surface parallel to $S_1$.  These Heegaard splittings are shown in the second picture in Figure~\ref{firststepfig}.  The third and fourth pictures suggest why $\Sigma'$ is an amalgamation of these Heegaard splittings for the components of $M \setminus S_1$.  The details of this reverse construction are left to the reader.
\begin{figure}[htb]
  \begin{center}
  \includegraphics[width=4.5in]{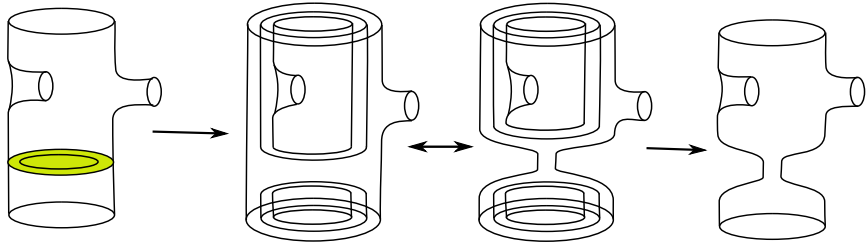}
  \caption{After compressing $S_0$ down to $S_1$, we can construct Heegaard splittings for the complementary components such that amalgamating along $S_1$ produces a surface isotopic to $S_0$.}
  \label{firststepfig}
  \end{center}
\end{figure}

For $i > 0$, let $M^*_i$ and $M^\#_i$ be the closures of the components of $M \setminus S_i$.  Assume $\Sigma'$ is an amalgamation along $S_i$ of Heegaard splittings $(\Sigma^*_i,H^{*-}_i,H^{*+}_i)$ and $(\Sigma^\#_i,H^{\#-}_i,H^{\#+}_i)$ for $M^*$ and $M^\#$, respectively.  Let $D_i$ be the disk such that compressing $S_i$ along $D_i$ produces $S_{i+1}$.  This disk is contained in one of the components of $M \setminus S_i$ and we will assume, without loss of generality, that it is contained in $M^*$.

By Lemma 1.1 in~\cite{cass:red}, there is a sequence of isotopies and 1-surgeries (compressing along disks) of $D_i$ after which $D_i \cap \Sigma^*_i$ is a single loop.  Because $M$ is irreducible, the 1-surgeries do not change the isotopy class of $D_i$, and we can assume that $D_i$ has been isotoped to intersect $\Sigma^*_i$ in a single loop.  

After this isotopy, the disk $D_i$ intersects the compression body $H^{*+}_i$ in an annulus with one boundary loop in $\partial_- H^{*+}_i$ and the other boundary component in $\partial_+ H^{*+}_i$.  The intersection with the compression body $H^{*-}_i$ is a properly embedded, essential disk.  

Let $N$ be a regular neighborhood in $M^*_i$ of $D_i$ and let $N'$ be a regular neighborhood of the closure of $N$.  The intersection $H^{*-}_i \cap N'$ is a regular neighborhood of a properly embedded essential disk so $H^{*-}_{i+1} = H^{*-}_i \setminus N'$ is a compression body.  Moreover the set $H^{*+}_{i+1} = (H^{*+}_i \cup \bar N') \setminus N$ is a compression body that we can think of as compressing the surface cross interval part of $H^{*+}_i$ across $D_i$, as in Figure~\ref{compressokfig}.  The negative boundary of $H^{*+}_{i+1}$ is isotopic to $S_{i+1}$ so we have constructed a Heegaard splitting for one of the components of $M \setminus S_{i+1}$.
\begin{figure}[htb]
  \begin{center}
  \includegraphics[width=3.5in]{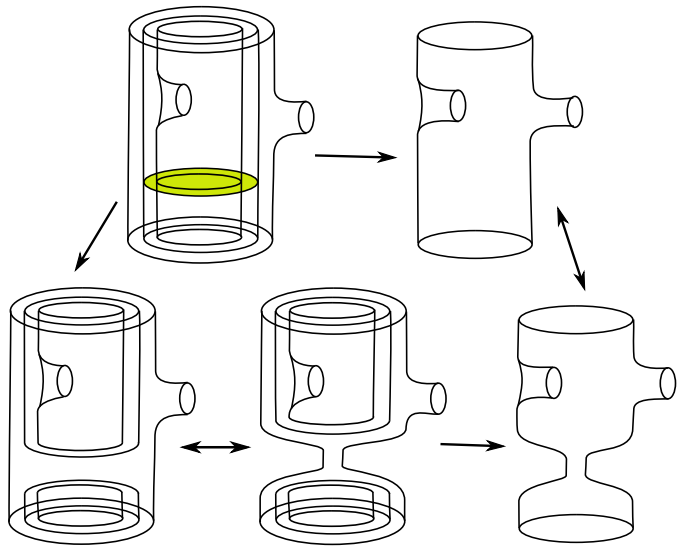}
  \caption{Compressing $S_i$ along a disk that intersects $\Sigma^*_i$ in a single loop suggests a new pair of Heegaard splittings that amalgamate to $\Sigma'$.}
  \label{compressokfig}
  \end{center}
\end{figure}

To construct a Heegaard splitting of the other complementary component, let $H^{\#-}_{i+1} = H^{\#-}_i$ and let $H^{\#+}_{i+1} = H^{\#+}_i \cup \bar N$.  The first is a compression body by definition.  The second is the result of gluing a 2-handle into the negative boundary of a compression body.  The resulting set is also a compression body so we have constructed a Heegaard splitting for the second component of $M \setminus S_{i+1}$, as in Figure~\ref{compressokfig}.  The reader can check that amalgamating these two Heegaard splittings produces $(\Sigma', H'^-, H'^+)$.
\end{proof}

\begin{Lem}
\label{crossisepsurfacelem}
Let $F$ be a closed surface and $S \subset F \times (a,b)$ a compact, closed, embedded, two-sided surface (not necessarily connected) that separates $F \times \{0\}$ from $F \times \{1\}$.  Then there is a sequence of surfaces $S = S_0,S_1,\dots,S_n$ such that each $S_{i+1}$ results from compressing $S_i$ across a disk in $F \times [a,b]$ and $S_n$ is a collection of spheres and one or more horizontal surfaces isotopic to $F \times \{c\}$ for some $c \in (a,b)$.
\end{Lem}

\begin{proof}
Write $S_0 = S$.  If $S_0$ is compressible in $F \times [a,b]$ then let $S_1$ be the result of compressing $S$ along an essential disk contained in $F \times [a,b]$.  Because the compression is disjoint from $F \times \{a\}$ and $F \times \{b\}$, the surface $S_1$ also separates $F \times \{a\}$ and $F \times \{b\}$.  We can repeat the process until we produce an incompressible surface $S_n$ that separates $F \times \{a\}$ and $F \times \{b\}$.

The only closed incompressible surfaces in $F \times [-1,1]$ are spheres and surfaces isotopic to $F \times \{c\}$ for some $c \in (a,b)$.  Thus if $S_n$ is incompressible then each non-sphere component of $S_n$ is isotopic to some $F \times \{c\}$.  If $S_n$ is a collection of spheres then $S_n$ cannot separate $F \times \{a\}$ from $F \times \{b\}$.  Thus $S_n$ has at least one component isotopic to $F \times \{c\}$.  
\end{proof}

\begin{Lem}
\label{onesthreetslem}
If  $g$ spans $f$ positively and $g$ spans $f$ negatively then there is a value $s \in [-1,1]$ and values $t_- < t_0 < t_+ \in [-1,1]$ such that either 
\begin{enumerate}
\item $\Sigma_{t_-}$ and $\Sigma_{t_+}$ are mostly above $\Sigma'_s$ while $\Sigma_{t_0}$ is mostly below $\Sigma'_s$ or 
\item $\Sigma_{t_-}$ and $\Sigma_{t_+}$ are mostly below $\Sigma'_s$ while $\Sigma_{t_0}$ is mostly above $\Sigma'_s$
\end{enumerate}
\end{Lem}

This Lemma should seem obvious from the graphic shown in the bottom right corner of Figure~\ref{spanningfig} for a sweep-out spanning in both directions.  Nonetheless, we will provide a proof just to be safe.

\begin{proof}
Because $g$ spans $f$ positively, there are values $s^p$ and $t^p_- < t^p_+$ such that $\Sigma_{t^p_-}$ is mostly below $\Sigma'_{s^p}$ while $\Sigma_{t^p_+}$ is mostly above $\Sigma'_{s^p}$.  Since $g$ spans $f$ negatively, there are values $s^n$ and $t^n_- > t^n_+$ such that $\Sigma_{t^n_-}$ is mostly below $\Sigma'_{s^n}$ while $\Sigma_{t^n_+}$ is mostly above $\Sigma'_{s^n}$.  It is either the case that $t^p_+$ or $t^n_+$ is between $t^p_-$ and $t^n_-$ or vice versa.  Without loss of generality, assume $t^p_+$ is between $t^p_-$ and $t^n_-$, i.e. $t^p_- < t^p_+ < t^n_-$.

First assume $s_p \geq s_n$.  Since $\Sigma_{t^n_-}$ is mostly below $\Sigma'_{s^n}$, it is also mostly below $\Sigma'_{s^p}$.  Thus the values $s = s^p$, $t_- = t^p_-$, $t_0 = t^p_+$ and $t_+ = t^n_+$ satisfy the criteria for the lemma.

Otherwise, assume $s^p < s^n$.  If $t^n_+ > t^p_-$ then $t^p_- < t^n_+ < t^n_-$ and we have the previous case, but with $p$ and $n$ reversed.  As in the previous case, there are values that satisfy the lemma.  Otherwise, we have $t^n_+ < t^p_- < t^p_+ < t^n_-$.  Focussing on $t^n_+ < t^p_- < t^p_+$, we have the first case, but with $+$ and $-$ reversed.  We can again find values that satisfy the lemma.
\end{proof}

\begin{proof}[Proof of Lemma~\ref{bothfacinglem}]
We will describe the case when $g$ spans $f$ both positively and negatively.  The case when $g$ spans $f$ with just one sign follows along the same lines, but is simpler and will thus be left for the reader.

Assume $g$ spans $f$ both positively and negatively.  By Lemma~\ref{onesthreetslem}, there are values $s$ and $t_- < t_0 < t_+ \in [-1,1]$ such that $\Sigma_{t_-}$ and $\Sigma_{t_+}$ are mostly above $\Sigma'_s$ while $\Sigma_{t_0}$ is mostly below $\Sigma'_s$, or vice versa.  Without loss of generality, assume $\Sigma_{t_-}$ and $\Sigma_{t_+}$ are mostly above $\Sigma'_s$.  

The surface $\Sigma'_s$ intersects each of the surfaces $\Sigma_{t_-}$, $\Sigma_{t_0}$, $\Sigma_{t_+}$ in loops that are trivial in the respective level surfaces of $f$.  Let $S$ be the result of isotoping $\Sigma'_s$ to remove any intersections that are trivial in both surfaces, and compressing $\Sigma'_s$ along innermost disks in $\Sigma_{t_-}$, $\Sigma_{t_0}$, $\Sigma_{t_+}$ whose boundaries are essential in $\Sigma'_s$ until $S$ is disjoint from $\Sigma_{t_-}$, $\Sigma_{t_0}$ and $\Sigma_{t_+}$.  Define $\Sigma = S_0,S_1,\dots,S_n = S$ to be the sequence of surfaces defined by these compressions.

Define $M^-_0 = H^-$ and $M^+_0 = H^+$.  After each compression, we can split the components of $M \setminus S_{i+1}$ into sets $M^-_{i+1}$ and $M^+_{i+1}$ in a unique way induced from the labeling for the complement of $S_i$.  That is, we will let $M^-_{i+1}$ be the union of the components consisting of $M^-_i$ plus or minus the neighborhood of $D_i$, and let $M^+_{i+1}$ be the union of $M^+_i$ plus or minus the neighborhood of $D_i$.  

After all the compressions, the surfaces $\Sigma_{t_-}$ and $\Sigma_{t_+}$ are contained in $M^+_n$ while $\Sigma_{t_0}$ is in $M^-_n$.  Thus the intersection of $S$ with $f^{-1}([t_-,t_0])$ separates $\Sigma_{t_-}$ from $\Sigma_{t_0}$ and the intersection of $S$ with $f^{-1}([t_0,t_+])$ separates $\Sigma_{t_0}$ from $\Sigma_{t_+}$.

By Lemma~\ref{crossisepsurfacelem}, we can compress $S$ further to a surface $S'$ whose intersection with $f^{-1}([t_-,t_+])$ is a union of spheres and at least two components isotopic to $\Sigma_{t_0}$.  By Lemma~\ref{compresstoamalglem}, $\Sigma'$ is an amalgamation along this $S'$.  If $S'$ consists exactly of two copies of $\Sigma$ then we're done.  

Otherwise, let $S'' \subset S'$ be the union of two components isotopic to $\Sigma$.  This $S''$ is separating and the generalized Heegaard splitting we constructed for $S'$ induces a generalized Heegaard splitting for each component of $M \setminus S''$.  Amalgamate these generalized Heegaard splittings to form a Heegaard splitting for each component of $M \setminus S''$.  Since $(\Sigma', H'^-, H'^+)$ is an amalgamation of the original generalized splitting, it is also an amalgamation of these new Heegaard splittings along the two copies of $\Sigma$ forming $S''$.

We can calculate the bound $k' \geq 2k$ by noting that a sequence of compressions turned $\Sigma'$ into the surface $S'$ containing two or more genus $k$ surfaces, so $\Sigma'$ must have genus at least $2k$.
\end{proof}

\section{Splitting sweep-outs}
\label{boundsect}

Let $f$ and $g$ be sweep-outs for a 3-manifold $M$ and assume $f \times g$ is generic.  As above, let $R_a \subset [-1,1] \times [-1,1]$ be the set of values $(t,s)$ such that $\Sigma_t$ is mostly above $\Sigma'_s$.  Let $R_b \subset [-1,1] \times [-1,1]$ be the set of values $(t,s)$ such that $\Sigma_t$ is mostly below $\Sigma'_s$.  

\begin{Def}
If $f \times g$ is generic and no arc $[-1,1] \times \{s\} \subset [-1,1] \times [-1,1]$ passes through both $R_a$ and $R_b$ then we will say that $g$ \textit{splits} $f$.  If two Heegaard splittings $(\Sigma, H^-, H^+)$ and $(\Sigma', H'^-, H'^+)$ for $M$ are represented by sweep-outs $f$ and $g$, respectively such that $g$ splits $f$ then we will say that $\Sigma'$ splits $\Sigma$.
\end{Def}

Note that by the definitions of spanning and splitting, if $f \times g$ is generic then either $g$ spans $f$ or $g$ splits $f$.  A picture of the graphic for a pair of splitting sweep-outs is shown in the bottom left corner of Figure~\ref{spanningfig}.  Let $k$ and $k'$ be the genera of $\Sigma$ and $\Sigma'$, respectively.  Recall that $d(\Sigma)$ is the Hempel distance of the Heegaard splitting $(\Sigma, H^-, H^+)$.  In this section we will prove the following:

\begin{Lem}
\label{neitherfacinglem}
If $\Sigma'$ splits $\Sigma$ then $k' \geq \frac{1}{2} d(\Sigma)$.
\end{Lem}

It may be easier to think of the inequality $k' \geq \frac{1}{2} d(\Sigma)$ as $d(\Sigma) \leq 2k'$.  This is more reminiscent of the inequality found by Scharlemann and Tomova~\cite{tom:dist}, and comes from a very similar argument.  In fact, combining  Lemma~\ref{bothfacinglem} and Lemma~\ref{neitherfacinglem} provides a new proof of the Heegaard splitting case of Scharlemann and Tomova's theorem.

\begin{Coro}[Scharlemann and Tomova~\cite{tom:dist}]
\label{stcoro}
Let $\Sigma$ and $\Sigma'$ be Heegaard surfaces in the same 3-manifold.  Let $k'$ be the genus of $\Sigma'$.  Then either $k' \geq \frac{1}{2} d(\Sigma)$, $\Sigma'$ is a stabilization $\Sigma$ or $\Sigma'$ is a stabilization of a boundary stabilization of $\Sigma$.
\end{Coro}

\begin{proof}
Let $f$ and $g$ be sweep-outs for $\Sigma$ and $\Sigma'$, respectively.  Isotope $f$ and $g$ so that $f \times g$ is generic.  Then either $g$ splits $f$, in which case by Lemma~\ref{neitherfacinglem}, $k' \geq \frac{1}{2} d(\Sigma)$, or $g$ spans $f$, in which case by Lemma~\ref{bothfacinglem}, $\Sigma'$ is an amalgamation along $\Sigma$.  By the classification of Heegaard splittings of compression bodies~\cite[Corollary 2.12]{st:crossi}, every amalgamation along $\Sigma$ is either a stabilization $\Sigma$ or a stabilization of a boundary stabilization of $\Sigma$.
\end{proof}

As pointed out in~\cite{inflects}, a horizontal tangency in the graphic corresponds to a critical point in the function $g$.  Since $g$ is a sweep-out, it has no critical points away from its spines, so there can be no horizontal tangencies in the interior of the graphic.  Thus the maxima of the upper boundary of $R_b$ and minima of the lower boundary of $R_a$ are vertices of the graphic.  Let $C$ be the complement in $\{0\} \times (-1,1)$ of the projections of $R_a$ and $R_b$.  This is a (possibly empty) closed interval.  Because $f \times g$ is generic, if $C$ is a single point, $C = \{s\}$, then the arc $[-1,1] \times \{s\}$ must pass through a vertex of the graphic that is a maximum of $\bar R_a$ and a mimimum of $\bar R_b$.  Let $(t, s)$ be the coordinates of this vertex.  

For arbitrarily small $\epsilon$, the restriction of $g$ to $\Sigma_{t+\epsilon}$ is a Morse function.  Moreover, there are two consecutive critical points in the restriction such that each component of the subsurface below any level set below the first saddle is contained in a disk while each component of any subsurface above a level set above the second saddle is contained in a disk.  This is only possible in a torus.  

Since we assumed $\Sigma$ has genus at least two, the set $C$ must have more than one point.  Since there are finitely many vertices in the graphic and infinitely many points in $C$, there is an $s \in C$ such that the arc $[-1,1] \times \{s\}$ does not pass through a vertex of the graphic.  

\begin{Lem}
\label{nicesmorselem}
If $g$ splits $f$ then there is an $s$ such that $[-1,1] \times \{s\}$ is disjoint from $R_a$ and $R_b$ and the restriction of $f$ to $\Sigma'_s$ is a Morse function such that each level set in $\Sigma'_s$ contains a loop that is essential in the corresponding level set of $f$.
\end{Lem}

\begin{proof}
As above, we can choose $s$ such that $[-1,1] \times \{s\}$ is disjoint from the vertices of the graphic and from $R_a$ and $R_b$.  The restriction of $f$ to $\Sigma'_s$ is Morse because $[-1,1] \times \{s\}$ does not pass through any vertices of the graphic.  Each level set of the restriction is a collection of level sets in some $\Sigma_t$ that bound the intersection of $\Sigma_t$ with $H'^-_s$ and with $H'^+_s$.  Since $\Sigma_t$ is neither mostly above nor mostly below $\Sigma'_s$, these loops cannot all be trivial in $\Sigma_t$.  Thus the level set contains a loop that is essential in $\Sigma_t$.
\end{proof}

To simplify the notation, we will assume (by isotoping if necessary) that $\Sigma' = \Sigma_s$ for this value of $s$.

If $d(\Sigma) \leq 2$ then the Lemma follows immediately, since we assumed $\Sigma'$ has genus at least 2.  Thus we will assume $d(\Sigma) > 2$.  Bachman and Schleimer~\cite[Claims 6.3 and 6.7]{bsc:bndls} showed that in this case, there is some non-trivial interval $[a,b] \subset [-1,1]$ such that for $t \in [a,b]$, every loop of $\Sigma_t \cap \Sigma'_s$ that is trivial in $\Sigma'_s$ is trivial $\Sigma_t$.  

Let $a'$ be a regular level of $f|_{\Sigma'}$ just above $a$ and let $b'$ be a regular level just below $b$.  Since $a'$ is in the interval $(a,b)$, every component of $\Sigma' \cap \Sigma_{a'}$ that is trivial in $\Sigma'$ is trivial in $\Sigma_{a'}$.  The same is true for $\Sigma_{b'}$.

An innermost such loop in $\Sigma'$ bounds a disk disjoint from $\Sigma_{a'}$ and a second disk in $\Sigma_{a'}$.  Since $d(\Sigma) > 0$, $M$ is irreducible and the two disks cobound a ball.  Isotoping the disk in $\Sigma'$ across this ball removes the trivial intersection.  By repeating this process with respect to $\Sigma_{a'}$ and $\Sigma_{b'}$, we can produce a surface $\Sigma''$ isotopic to $\Sigma'$ such that each loop $\Sigma'' \cap \Sigma_{a'}$ and $\Sigma'' \cap \Sigma_{b'}$ is essential in $\Sigma''$.  Note that this has not changed the property that each regular level set of $f|_{\Sigma''}$ contains a loop that is essential in $\Sigma_t$.

Let $S$ be the intersection of $\Sigma''$ with $f^{-1}([a',b'])$.  Consider a projection map $\pi$ from $f^{-1}([a',b'])$ onto $\Sigma_0$.  The image of a level loop of $f|_{\Sigma'}$ under $\pi$ is a simple closed curve in $\Sigma_0$.  (Its isotopy class is well defined, even though its image depends on the choice of projection.)

\begin{Lem}
\label{parallelparallellem}
If two levels loop of $f|_{\Sigma'}$ are isotopic in $S$ then their projections are isotopic in $\Sigma_0$.
\end{Lem}

\begin{proof}
Any two level loops are disjoint in $S$ so if two level loops are isotopic then they bound an annulus $A \subset S$.  The projection of $A$ into $\Sigma_0$ determines a homotopy from one boundary of the image of $A$ to the other.  Thus the projections of the two loops are homotopic in $\Sigma_0$.  Homotopic simple closed curves in surfaces are isotopic so the two projections are in fact isotopic.
\end{proof}

Let $L$ be the set of all isotopy classes of level loops of $f|_S$.  These loops determine a pair-of-pants decomposition for $S$.  We will define a map $\pi_*$ from $L$ to the disjoint union $C(\Sigma_0) \cup \{0\}$ as follows:  A representative of a loop $\ell \in L$ projects to a simple closed curve in $\Sigma_0$.  If the projection is essential then we define $\pi_*(\ell)$ to be the corresponding vertex of $C(\Sigma)$.  If the projection is trivial then we define $\pi_*(\ell) = 0$.  By Lemma~\ref{parallelparallellem}, $\pi_*$ is well defined.  

\begin{Lem}
\label{pantsdsjtlem}
If $\ell$ and $\ell'$ are cuffs of the same pair of pants in the complement $S \setminus L$ then their images in $\Sigma_0$ are isotopic to disjoint loops.
\end{Lem}

\begin{proof}
Let $\ell, \ell', \ell'' \in L$ be three loops bounding a pair of pants in $S \setminus L$.  There is a saddle singularity in $f|_{\Sigma'}$ contained in a level component $E$ (a graph with one vertex and two edges) such that $\ell$, $\ell'$ and $\ell''$ are isotopic to the boundary loops of a regular neighborhood of $E$.

The projection of $E$ into $\Sigma_0$ is a graph $\pi(E)$ with one vertex and two edges.  The projections of the level loops near $E$ define a homotopy from the projections of representatives of $\ell$, $\ell'$, $\ell''$ into $\pi(E)$.  Since these representatives are simple in $\Sigma_0$, they must be isotopic to the boundary components of a regular neighborhood of $\pi(E)$.  Thus $\pi_*(\ell)$ is disjoint from $\pi_*(\ell')$.
\end{proof}

Thus if $\ell$ and $\ell'$ are cuffs of the same pair of pants and their projections are essential in $\Sigma_0$ then $\pi_*(\ell)$ and $\pi_*(\ell')$ are connected by an edge in $C(\Sigma_0)$.  Define $L' = \pi_*(L) \cap C(\Sigma)$.

\begin{Lem}
\label{lconnectedlem}
The set $L'$ is connected and has diameter at most $2k' - 2$.
\end{Lem}

\begin{proof}
For each regular value $t \in (a,b)$ of $f|_S$, let $L_t \subset L$ be the set of loops with representatives in $(f|_S)^{-1}(t)$.  The loops in $L_t$ are pairwise disjoint so their projections in $\Sigma_0$ are pairwise disjoint.  Moreover, the projection $L'_t = \pi_*(L_t) \cap C(\Sigma)$ contains at least one essential loop, so $L'_t$ is a non-empty simplex in $C(\Sigma_0)$.  If there are no critical points of $f|_S$ between $t$ and $t'$ then the level sets are isotopic, so $L_t = L_{t'}$ and $L'_t = L'_{t'}$.

If there is a single critical point in $f|_S$ between $t$ and $t'$ then $L_t$ may be different from $L_{t'}$.  If the critical point is a central singularity (a maximum or a minimum) then the difference between the level sets is a trivial loop in $\Sigma'$, so $L'_t = L'_{t'}$.  If the critical point is a saddle then the projections of $L_t$ are pairwise disjoint from the projections of $L_{t'}$.  Thus for any values $t, t' \in [a',b']$, there is a path in $L'$ from any vertex of $L'_t$ to any vertex in $L'_{t'}$.  Since $L'$ is the union of all the sets $\{L'_t | t \in [a',b']\}$, $L'$ is connected.

Consider loops $\ell, \ell' \in L$ whose projections are essential in $\Sigma_0$.  Since $L'$ is connected, there is a path $\pi_*(\ell) = v_0, v_1,\dots,v_n = \pi_*(\ell')$ in $C(\Sigma)$.  Assume we have chosen the shortest such path.  Each $v_i$ is the projection of a loop $\ell_i \in L$.  If $\ell_i$ and $\ell_j$ are cuffs of the same pair of pants in $S \setminus L$ then $v_i$ and $v_j$ are distance one in $C(\Sigma_0)$.  Since the path is minimal, $i$ and $j$ must be consecutive.  Thus there is at most one step in the path for each pair of pants in $S \setminus L$.  

The number of pairs of pants is at most the negative Euler characteristic of $S$.  Since $\partial S$ is essential in $\Sigma''$, the Euler characteristic of $S$ is greater (less negative) than or equal to that of $\Sigma''$.  The Euler characteristic of $\Sigma''$ is $2 - 2k'$ so the path from $\pi_*(\ell)$ to $\pi_*(\ell')$ has length at most $2k' - 2$.
\end{proof}

\begin{proof}[Proof of Lemma~\ref{neitherfacinglem}]
Assume $(\Sigma', H'^-, H'^+)$ splits $(\Sigma, H^-, H^+)$.  Let $[a, b] \subset [-1,1]$ be the largest interval such that for $t \in [a,b]$, every loop of $\Sigma_t \cap \Sigma'_s$ that is trivial in $\Sigma'_s$ is trivial $\Sigma_t$.   Let $a',b' \in [-1,1]$ be just inside $[a,b]$ as defined above.  Isotope $\Sigma'$, as described, to a surface $\Sigma''$ such that $S = \Sigma'' \cap f^{-1}([a',b'])$ has essential boundary in $\Sigma''$ and each level set $(f|_S)^{-1}(t)$ contains an essential loop in $\Sigma_t$ for $t \in [a',b']$.  

For small enough $t$, the level loops of $f|_{\Sigma''}$ bound disks in $\Sigma''$.  At least one of these loops projects to an essential loop in $\Sigma_0$ so $a > 0$.  The value $a$ is a critical level of $f|_\Sigma''$ containing a saddle singularity.  As above, the projections of the level loops before and after this essential saddle are pairwise disjoint.  

By the definition of $a$, the projection of the level loops before the saddle contain a vertex of $\mathcal{H}^-$.  The projection of the level set after $a$ is contained in $L'$ so $d(\mathcal{H}^-,L') = 1$.  A parallel argument implies $d(\mathcal{H}^+,L') = 1$.  By Lemma~\ref{lconnectedlem}, the set $L'$ of projections of level loops into $\Sigma_0$ is connected and has diameter at most $2k'-2$.  Thus $d(\Sigma) \leq 2k'$.
\end{proof}

\section{Isotopies of sweep-outs}
\label{isotopysect}

If $(\Sigma', H'^-, H'^+)$ is flippable and spans $(\Sigma, H^-, H^+)$ then it will span $(\Sigma, H^-, H^+)$ both positively and negatively.  In particular it will be represented by one sweep-out that spans a sweep-out for $\Sigma$ positively and another that spans a sweep-out for $\Sigma$ negatively.  These sweep-outs will be isotopic and we would like to understand how the graphic changes during this isotopy.

\begin{Lem}
\label{isotopesweepslem}
Let $g$ and $g'$ be sweep-outs such that $f \times g$ and $f \times g'$ are generic and $g$ is isotopic to $g'$.  Then there is a family of sweep-outs $\{g_r | r \in [0,1]\}$ such that $g = g_0$, $g' = g_1$ and for all but finitely many $r \in [0,1]$, the graphic defined by $f$ and $g_r$ is generic.  At the finitely many non-generic points, there are at most two valence two or four vertices at the same level, or one valence six vertex.
\end{Lem}

The analogous Lemma for isotopies of Morse functions is Lemma 9 in~\cite{me:t3} and Lemma~\ref{isotopesweepslem} can be proved by a similar argument.  We will allow the reader to work out the details.

As above, let $(\Sigma, H^-, H^+)$ and $(\Sigma', H'^-, H'^+)$ be Heegaard splittings for a 3-manifold $M$ with genera $k$ and $k'$, respectively.

\begin{Lem}
\label{sigmabothlem}
If $(\Sigma', H'^-, H'^+)$ spans $(\Sigma, H^-, H^+)$ both positively and negatively then $k' \geq \min \{\frac{1}{2} d(\Sigma), 2k\}$.
\end{Lem}

\begin{proof}
Since $(\Sigma', H'^-, H'^+)$ spans $(\Sigma, H^-, H^+)$ both positively and negatively, there are sweep-outs $f$, $g$ representing $(\Sigma, H^-, H^+)$ and $(\Sigma', H'^-, H'^+)$, respectively, such that $g$ spans $f$ positively, as well as sweep-outs $f'$, $g'$ representing the two Heegaard splittings such that $g'$ spans $f'$ negatively.  

Since $f$ and $f'$ represent the same Heegaard splittings, there is a sequence of handle slides after which there is an isotopy taking $f'$ to $f$.  The handle slides can be done so that before the isotopy, $g'$ still spans $f'$.  By composing $g'$ with this isotopy, we can assume $g'$ spans $f$ negatively.  Because $g$ and $g'$ represent the same Heegaard splitting, they will be isotopic after an appropriate sequence of handle slides that do not change the fact that $g'$ spans $f$ negatively.

Consider a continuous family of sweep-outs $\{g_r |  r \in [0,1], g_r \in C^{\infty}(M,\mathbf{R})\}$ such that $g_0 = g$, $g_1 = g'$ and $f \times g_r$ is generic for all but finitely many $r$.  For a generic $r$, $g_r$ either spans $f$ or splits $f$.  If $g_r$ spans $f$ with both signs or splits $f$ then by Lemmas~\ref{bothfacinglem} and~\ref{neitherfacinglem}, $k' \geq \min \{\frac{1}{2} d(\Sigma), 2k\}$.  Thus away from the finitely many non-generic values, we will assume for contradiction that $g_r$ spans $f$ positively or negatively, but not both.

Since $g_0$ spans $f$ positively and $g_1$ spans $f$ negatively, there must be some value $r_0$ such that for small $\epsilon > 0$, $g_{r_0-\epsilon}$ spans $f$ positively, while $g_{r_0+\epsilon}$ spans $f$ negatively.  For every small $\epsilon > 0$, the closures of the projections of $R_a$ and $R_b$ at time $r_0 - \epsilon$ intersect in an interval $I^-_\epsilon$.  Since the projections are disjoint at time $r_0$, the limit of these intervals must contain a single point $s^-$.  Thus the graphic at time $r_0$ must have two vertices at the same level, one of which is a maximum for the upper boundary of $R_a$ and the other a minimum for the lower boundary of $R_b$, as in the middle graphic shown in Figure~\ref{changefig}.
\begin{figure}[htb]
  \begin{center}
  \includegraphics[width=3.5in]{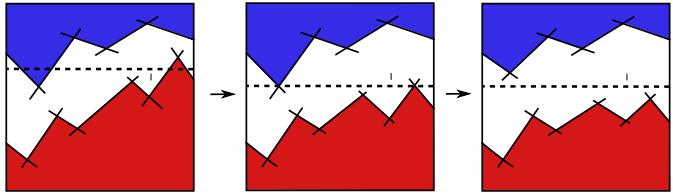}
  \caption{When the graphic goes from spanning positively to not spanning positively, there are two vertices at the same level.}
  \label{changefig}
  \end{center}
\end{figure}

If the vertices in the upper boundary of $R_a$ and the lower boundary of $R_b$ coincide, then this vertex cannot be valence four, as explained above, since $\Sigma$ is not a torus.  The same argument implies that this cannot happen at a valence six vertex either.  Since $g_{r_0 - \epsilon}$ spans $f$ positively, the $s$ coordinate of the vertex in the boundary of $R_a$ must be strictly lower than that the vertex in the boundary of $R_b$.  However, an analogous argument for the graphics at times $r_0 + \epsilon$ implies that the $s$ coordinate of the vertex in the boundary of $R_a$ must be strictly greater than that of the vertex in the boundary of $R_b$.  Since there are at most two vertices at level $s$, this is a contradiction and completes the proof.
\end{proof}

\begin{proof}[Proof of Theorem~\ref{mainthm1}]
By Lemma~\ref{withitselflem}, $(\Sigma, H^-, H^+)$ spans itself positively. By Lemma~\ref{stabfacinglem}, every stabilization of $(\Sigma, H^-, H^+)$ spans $(\Sigma, H^-, H^+)$ positively.  If $(\Sigma', H'^-, H'^+)$ is a flippable stabilization of $(\Sigma, H^-, H^+)$ then $(\Sigma', H'^-, H'^+)$ spans $(\Sigma, H^-, H^+)$ both positively and negatively and by Lemma~\ref{sigmabothlem}, $\Sigma'$ has genus greater than or equal to $\min \{\frac{1}{2} d(\Sigma), 2k\}$.
\end{proof}

\begin{proof}[Proof of Theorem~\ref{mainthm2}]
Let $(\Sigma, H^-, H^+)$ be a Heegaard splitting of a 3-manifold $M$ with a single boundary component.  Recall the convention that for any Heegaard splitting for $M$, the first compression body is a handlebody.  This implies that if $(\Sigma', H'^-, H'^+)$ is a Heegaard splitting for $M$ and $\Sigma'$ is isotopic to $\Sigma$ (as an unoriented surface) then $(\Sigma', H'^-, H'^+)$ is isotopic to $(\Sigma, H^-, H^+)$.  

As in the last proof, $(\Sigma, H^-, H^+)$ spans itself positively, as does every stabilization of $(\Sigma, H^-, H^+)$.  Let $(\Sigma', H'^-, H'^+)$ be a boundary stabilization of $(\Sigma, H^-, H^+)$.  By Lemma~\ref{bstabfacinglem}, $(\Sigma', H'^-, H'^+)$ spans $(\Sigma, H^-, H^+)$ negatively, as does every stabilization of $(\Sigma', H'^-, H'^+)$. Any common stabilization of $\Sigma$ and $\Sigma'$ spans $\Sigma$ with both signs so by Lemma~\ref{sigmabothlem}, every common stabilization has genus at least $\min \{\frac{1}{2} d(\Sigma), 2k\}$.
\end{proof}

\begin{proof}[Proof of Corollary~\ref{maincoro}]
Let $(\Sigma, H^-, H^+)$ be a Heegaard splitting of a 3-manifold $M$ with a single boundary component.  Assume $\Sigma$ has minimal genus and $d(\Sigma) > 2(k + b)$ where $k$ is the genus of $\Sigma$ and $b$ is the genus of $\partial M$.  Because $(\Sigma', H'^-, H'^+)$ is a boundary stabilization, it is weakly reducible.  Its genus is $k + b$ and the Heegaard genus of $M$ is $k$ so $\Sigma'$ is not minimal genus.  Assume for contradiction $\Sigma'$ is stabilized. 

Then $\Sigma'$ is a stabilization of a Heegaard surface $\Sigma''$ of genus strictly less than $k + b$.  By Scharlemann and Tomova's theorem~\cite{st:dist} (See also Corollary~\ref{stcoro}), every Heegaard splitting of $M$ of genus less than $\frac{1}{2} d(\Sigma)$ is either a stabilization of $\Sigma$ or a stabilization of a boundary stabilization of $\Sigma$.  Every boundary stabilization of $\Sigma$ has genus at least $k + b$ so such a $\Sigma''$ must be a stabilization of $\Sigma$.  Thus if $\Sigma'$ is stabilized then it is a stabilization of $\Sigma$.

Since $(\Sigma', H'^-, H'^+)$ spans $(\Sigma, H^-, H^+)$ negatively, Theorem~\ref{mainthm2} implies that any common stabilization of $\Sigma$ and $\Sigma'$ has genus strictly greater than $k + b$.  If $\Sigma'$ were a stabilization of $\Sigma$ then it would be a common stabilization so $\Sigma'$ is not a stabilization of $\Sigma$.  This contradiction implies that $\Sigma'$ is irreducible.
\end{proof}

\bibliographystyle{amsplain}
\bibliography{stabs}

\end{document}